\documentclass[12pt]{amsart}

\usepackage{verbatim}
\usepackage{amsfonts,amscd,latexsym,amsmath,amssymb,enumerate,verbatim,amsthm,epsfig,alltt,enumitem,color}
\newtheorem{theorem}{Theorem}[section]
\newtheorem{lemma}[theorem]{Lemma}
\newtheorem{proposition}[theorem]{Proposition}
\newtheorem{conjecture}[theorem]{Conjecture}
\newtheorem{corollary}[theorem]{Corollary}
\newtheorem*{claim}{Claim}
\newtheorem{observation}[theorem]{Observation}
\newtheorem{definition}[theorem]{Definition}
\newtheorem{fact}[theorem]{Fact}
\theoremstyle{remark}

\newtheorem{problem}[theorem]{Problem}

\newcommand{\Int}{\mathbb{Z}}
\newcommand{\Nat}{\mathbb{N}}

\newcommand{\U}{\mathcal{U}}
\newcommand{\G}{\mathcal{G}}
\newcommand{\E}{\mathcal{E}}

\newcommand{\N}{{\mathbb N}}

\usepackage{hyperref}
\usepackage{xcolor}
\definecolor{green}{RGB}{0,127,0}
\definecolor{red}{RGB}{105,89,205}
\DeclareMathOperator{\aut}{Aut}
\DeclareMathOperator{\id}{id}
\DeclareMathOperator{\supp}{supp}

\begin{document}
\title{On Dual surjunctivity and applications}
\author[M. Doucha]{Michal Doucha}
\address{Institute of Mathematics\\
Czech Academy of Sciences\\
\v Zitn\'a 25\\
115 67 Praha 1\\
Czech Republic}
\email{doucha@math.cas.cz}
\author[J. Gismatullin]{Jakub Gismatullin}
\address{Instytut Matematyczny Uniwersytetu Wroc{\l}awskiego, pl. Grunwaldzki 2/4, 50-384 Wroc{\l}aw, Poland \& Instytut Matematyczny Polskiej Akademii Nauk, ul. {\'S}niadeckich 8, 00-656 Warszawa, Poland}
\email{jakub.gismatullin@uwr.edu.pl}

\thanks{M. Doucha was supported by the GA\v{C}R project 19-05271Y and RVO: 67985840. J. Gismatullin is supported by the National Science Centre, Poland NCN grants no.  2014/13/D/ST1/03491 and 2017/27/B/ST1/01467.}
\begin{abstract}
We explore the dual version of Gottschalk's conjecture recently introduced by Capobianco, Kari, and Taati, and the notion of dual surjunctivity in general. We show that dual surjunctive groups satisfy Kaplansky's direct finiteness conjecture for all fields of positive characteristic. By quantifying the notions of injectivity and post-surjectivity for cellular automata, we show that the image of the full topological shift under an injective cellular automaton is a subshift of finite type in a quantitative way. Moreover we show that dual surjunctive groups are closed under ultraproducts, under elementary equivalence, and under certain semidirect products (using the ideas of Arzhantseva and Gal for the latter); they form a closed subset in the space of marked groups, fully residually dual surjunctive groups are dual surjunctive, etc. We also consider dual surjunctive systems for more general dynamical systems, namely for certain expansive algebraic actions, employing results of Chung and Li.
\end{abstract}

\maketitle

\tableofcontents

\section*{Introduction}
In the beginning of the 1970s, W. Gottschalk introduced the following notion. Let $G$ be a group and $A$ a finite set, and let us consider the topological Bernoulli shift $G\curvearrowright A^G$. If any injective, continuous, and $G$-equivariant map $T\colon A^G\rightarrow A^G$ is also surjective, then $G$ is called \emph{surjunctive}. Gottschalk asked in \cite{gott} whether every group is surjunctive. The question reached its prominence after Gromov proved in \cite{grom} that groups that were later going to be called sofic (see \cite{weiss}) are surjunctive. Sofic groups, originally introduced just because of Gottschalk's question, now live a life on their own and are one of the most important classes of groups in geometric group theory, topological and measurable dynamics, graph theory, and beyond. Peculiarly, as of now, there is still not known any non-sofic group, although the existence of such groups is rather generally expected.

A natural idea is to consider the reverse of the Gottschalk question. The perhaps most direct attempt, that is, asking whether every surjective, continuous, and $G$-equivariant map $T\colon A^G\rightarrow A^G$ is injective, is not the right choice. Indeed, there are counterexamples even for $G=\Int$ (we refer to \cite{csc_book} for examples of such kind). It is however instructive to recall at this point the notion of Garden of Eden. Although surjectivity of the map $T\colon A^G\rightarrow A^G$ does not necessarily imply injectivity, it does imply, in some cases, a weaker notion called pre-injectivity (we refer further to the paper for a definition, or to the monograph \cite{csc_book}). The theorems of Moore and Myhill (\cite{Moore}, \cite{Myhill}) established the equivalence of surjectivity and pre-injectivity for $G=\Int^d$. Later the same equivalence was obtained for all finitely generated groups of subexponential growth in \cite{MaMi}, and finally for all amenable groups in \cite{CSMaSc}. This line of research culminated rather recently when Bartholdi in the combination of the two papers \cite{Ba10} and \cite{Ba19} showed that the Garden of Eden equivalence characterizes the class of amenable groups. We remark that the Garden of Eden equivalence has been considered and proved for dynamical systems much more general than topological Bernoulli shifts. We refer to \cite{Li2019} and references therein for more information.

It was also very recently when Capobianco, Kari, and Taati found a proper reverse, or dual, of Gottschalk's question. In \cite{cap}, they introduce the notion of post-surjectivity, which is strictly stronger than surjectivity, and ask for which groups, rightfully called dual surjunctive, post-surjectivity implies pre-injectivity. As Gromov did for surjunctive groups, they show that all sofic groups are dual surjunctive. Since, as they show for topological Bernoulli shifts, post-surjectivity together with pre-injectivity actually implies injectivity, one is led to a strong version of the Garden of Eden theorem which says that a continuous $G$-equivariant map $T\colon A^G\rightarrow A^G$ is injective if and only if it is post-surjective. All sofic groups therefore satisfy this strong version of Garden of Eden.\medskip

The aim of this note is to explore the notions of post-surjectivity and dual surjunctivity further. We simplify some arguments from \cite{cap} concerning post-surjectivity and pre-injectivity, and we investigate these notions in a quantitative way. We also introduce and investigate a notion of post-surjectivity for more general shifts and even more general expansive dynamical systems. Below is a selection of some of our results.\medskip

\noindent{\bf Results.} 
\begin{enumerate}
    \item Dual surjunctive groups satisfy Kaplansky's direct finiteness conjecture for all fields of positive characteristic (see Theorem~\ref{thm:kap}). We also use the opportunity to consider a metric version of Kaplansky's conjecture (see Theorem~\ref{thm:kapp}).
    \item If $T\colon A^G\rightarrow A^G$ is an injective cellular automaton, then $T[A^G]$ is a subshift of finite type with memory set of the forbidden patterns precisely depending on the injectivity of $T$ (see Theorem~\ref{thm:injectivity-finitesubshift}).
    \item Dual surjunctive groups are closed under taking ultraproducts and under elementary equivalence, they form a closed subset in the space of marked groups, fully residually dual surjunctive groups are dual surjunctive (see Theorems~\ref{thm:dualsurjultraprod} and~\ref{thm:closedinmarkedgrps}, and Corollaries~\ref{cor:fullyresiduallyDS} and~\ref{cor:elementaryeqDS}).
    \item Algebraic expansive actions of any countable polycyclic-by-finite group (and under some additional conditions, of any amenable group) on compact metrizable abelian groups with completely positive entropy (with respect to the Haar measure on the compact group) are dual surjunctive (see Theorem~\ref{thm:expansivealgactionDS}).\bigskip
\end{enumerate}

\section{Post-surjectivity and pre-injectivity}
Throughout the paper, $G$ is a group and $A$ is a finite set having at least two elements. We topologize $A^G$ with the product topology, where $A$ is equipped with the discrete topology. When $G$ is countably infinite, which will be the most interesting case, $A^G$ is then obviously homeomorphic to the Cantor space. $G$ acts (by homeomorphisms) on $A^G = \{f\colon G\to A\}$ by \[g\cdot f(x) =  f(g^{-1}x), \text{ for }g,x\in G,\ f\in A^G.\] The corresponding dynamical system is called the \emph{topological Bernoulli shift}, or just \emph{topological (full)-shift}.

We need few more definitions from the dynamics on topological shifts. We refer the reader to \cite{csc_book} for a detailed treatment.
\begin{definition}
Let $G$ be a group and $A$ a finite set. Any element $x\in A^G$ is called a configuration. Any map $p\colon D\rightarrow A$, where $D\subseteq G$ is (usually finite, but not always) a subset, is called a \emph{pattern}. A pattern is called \emph{finite} if its domain is finite.
\end{definition}
\begin{definition}
Let $G$ be a group and $A$ a finite set. By a \emph{subshift}, we mean any closed subset $X\subseteq A^G$ that is also closed under the shift by the elements of $G$.

If $P\subseteq G$ is a subset, by $X_P$ we denote the set of patterns whose domain is $P$ and which are restrictions of configurations from $X$. That is, $X_P:=\{x\upharpoonright P\colon x\in X\}$.
\end{definition}
The following two types of subshifts will be the most interesting for us.
\begin{definition}
Let $X\subseteq A^G$ be a subshift. We say that
\begin{itemize}
    \item $X$ is \emph{of finite type} if there exists a finite set $\{p_1,\ldots,p_n\}$ of finite patterns such that for $x\in A^G$ we have $x\in X$ if and only if for no $g\in G$ and $i\leq n$ we have $g\cdot x\upharpoonright \mathrm{dom}(p_i)=p_i$.
    \item $X$ is \emph{strongly irreducible} if there exists a finite set $D\subseteq G$ such that for all finite patterns $p\colon P\rightarrow A$ and $p'\colon P'\rightarrow A$, with $p\in X_P$, resp. $p'\in X_{P'}$ and $P\cdot D\cap P'=\emptyset$ there exists $x\in X$ such that $x\upharpoonright P=p$ and $x\upharpoonright P'=p'$.
    
\end{itemize}
\end{definition}

By $T\colon A^G\rightarrow A^G$ we always mean a continuous $G$-equivariant map, that is \[T(g\cdot f) = g\cdot T(f).\] $T$ is called a \emph{cellular automaton} (further abbreviated \emph{CA}). It is well known that every such $T$ is induced by a map $\tau\colon A^F\rightarrow A$, where $F\subseteq G$ is a finite subset, called the \emph{memory set} for $T$,  such that for all $f\in A^G$, $x\in G$ the following holds true:
\[T(f)(x) = \tau(f\upharpoonright x\cdot F), \tag{$\ast$}\] where $f\upharpoonright x\cdot F$ is the pattern obtained by restricting $f$ to $x\cdot F:=\{xg\colon g\in F\}$.

Let us define an equivalence relation $\sim$ of \emph{almost equality} on $A^G$ in the following way, for $c,d\in A^G$ we write \[c\sim d \text{ if and only if }\{g\in G\colon c(g)\neq d(g)\}\text{ is finite.}\]
We record the following fact whose straightforward proof is left for the reader.
\begin{fact}
Let $X\subseteq A^G$ be a strongly irreducible subshift. Then for every $x\in X$, the equivalence class $[x]_\sim\cap X$ is dense in $X$.
\end{fact}

In \cite{cap}, Capobianco, Kari, and Taati introduced the following stronger version of surjectivity. We also recall below the by-now standard notion of pre-injectivity, a weaker version of injectivity.
\begin{definition}
\begin{enumerate}
\item \cite[Definition 2]{cap} $T$ is called  \emph{post-surjective} if whenever $T(g)\sim f'$, for any $g,f'\in A^G$, then there exists $g'\sim g$ such that $T(g')=f'$.

\item $T$ is called  \emph{pre-injective} if  whenever $f\sim f'$ and $T(f)=T(f')$, then $f=f'$, for all $f,f'\in A^G$.
\end{enumerate}
\end{definition}

\bigskip

For $f_1,f_2\in A^G$, we denote \[\Delta(f_1,f_2) = \{g\in G : f_1(g)\neq f_2(g)\}.\]

We shall also need a stronger version of post-surjectivity, which is what is actually useful in applications. It turns out that for full shifts, the two notions of post-surjectivity are equivalent.
\begin{definition}
Let $X,Y\subseteq A^G$ be subshifts. A CA $T\colon X\rightarrow Y$ is \emph{strongly post-surjective} if there exists a finite set $M\subseteq G$ such that for every $x\in X$ and $T(x)\sim z\in Y$ there exists $y\in X$ such that $y\sim x$, $T(y)=z$ and $\Delta(x,y)\subseteq \Delta(T(x),z)\cdot M$.

The set $M$ will be called a \emph{post-surjectivity set} for $T$.
\end{definition}

\begin{observation}\label{obs:postsurjectivity}
\begin{enumerate}
\item A CA $T\colon A^G\rightarrow A^G$ is post-surjective if and only if it is strongly post-surjective.
\item Let $X,Y\subseteq A^G$ be subshifts and let $T\colon X\rightarrow Y$ be a strongly post-surjective and pre-injective CA with a finite post-surjectivity set $M$. Then for every $c, d\in X$, $c\sim d$ the following holds true \[\Delta(c,d)\subseteq\Delta(T(c),T(d))\cdot M.\]
\end{enumerate}
\end{observation}
\begin{proof}
The non-trivial direction of $(1)$ is the content of \cite[Lemma 1]{cap}. For $(2)$, notice that by strong post-surjectivity there is some $d'\sim c$ with $T(d')=T(d)$ such that $\Delta(c,d')\subseteq\Delta(T(c),T(d))\cdot M$. By pre-injectivity, since $d'\sim d$ and $T(d)=T(d')$, we get that $d'=d$, and we are done.
\end{proof}

Our aim is now to strengthen of the main results from \cite{cap}, saying that pre-injective and (strongly) post-surjective CA on full shift is reversible (see \cite[Theorem 1]{cap}), and also to provide a simpler proof of it. First we observe that injectivity is equivalent to a kind of uniform injectivity.

\begin{proposition}\label{prop:unifinject}
Let $X\subseteq A^G$ be a closed invariant subshift. Let $T\colon X\rightarrow A^G$ be a cellular automaton. Then $T$ is injective if and only if there exists a finite subset $N\subseteq G$ such that for every $x,y\in X$, if $x(g)\neq y(g)$, for some $g\in G$, then there exists $h\in g\cdot N$ such that $T(x)(h)\neq T(y)(h)$.
\end{proposition}
\begin{proof}
Suppose first that $T$ is injective and the right side of the equivalence from the statement is not satisfied. We shall suppose that $G$ is countable, however the proof can be easily generalized to the uncountable case. Let $(N_m)_{m\in\N}$ be an increasing sequence of finite subsets of $G$ whose union covers $G$. For each $m$ there are $x_m,y_m\in X$ such that $x_m(g_m)\neq y_m(g_m)$, for some $g_m\in G$, however for every $h\in g_m\cdot N_m$ we have $T(x_m)(h)=T(y_m)(h)$. We can clearly assume that $g_m=1_G$. By compactness, we can assume that $x_m\to x$ and $y_m\to y$, when $m\to\infty$. By continuity of $T$ and since $\bigcup_{m\in\N} N_m=G$, we get that $T(x_m)\to T(x)$, $T(y_m)\to T(y)$, and $T(x)=T(y)$ although $x(1_G)\neq y(1_G)$. This is a contradiction with injectivity of $T$.

Conversely, suppose that $T$ satisfies the right side of the equivalence from the statement for a finite subset $N\subseteq G$, but it is not injective. Then there exist $x\neq y\in X$ such that $T(x)=T(y)$. In particular, there is $g\in G$ such that $x(g)\neq y(g)$, however for every $h\in g\cdot N$, $T(x)(h)=T(y)(h)$, a contradiction.
\end{proof}

\begin{definition}
Let $T\colon X\rightarrow A^G$ be an injective CA. The finite set $N\subseteq G$ guaranteed by Proposition~\ref{prop:unifinject} will be called an \emph{injectivity set} for $T$.
\end{definition}
\begin{lemma}\label{lem:post-surjective-surjective}
Let $X,Y\subseteq A^G$ be subshifts and $Y$ be strongly irreducible. Let $T\colon X\rightarrow Y$ be a strongly post-surjective CA. Then $T$ is surjective.
\end{lemma}
\begin{proof}
This is essentially proved in \cite[Proposition 2]{cap}. Since we work in a slightly more generality, we re-prove it for the convenience of the reader. Pick some $y\in Y$ and let $x\in X$ be arbitrary. By strong irreducibility, the equivalence classes in $\sim$ are dense in $Y$, so we can find a sequence $(y_n)_{n\in\Nat}$ such that $y_n\to y$ and $y_n\sim T(x)$, for every $n\in\Nat$. By strong post-surjectivity, there exists a sequence $(x_n)_{n\in\Nat}$ such that $x_n\sim x$ and $T(x_n)=y_n$, for every $n\in\Nat$. Let $x'\in X$ be a cluster point of this sequence. It is immediate that $T(x')=y$.
\end{proof}
\begin{theorem} \label{thm:isom}
Let $X,Y\subseteq A^G$ be subshifts and let $T$ be a strongly post-surjective and pre-injective CA. If $X$ is strongly irreducible, then $T$ is injective .

In particular, if both $X$ and $Y$ are strongly irreducible, then $T$ is an isomorphism.
\end{theorem}
\begin{proof}
Suppose that $X$ is strongly irreducible. By Proposition~\ref{prop:unifinject}, it suffices to check that there is a finite injectivity set $N\subseteq G$. Let $M\subseteq G$ be a post-surjectivity set for $T$. We claim that $M^2$ is an injectivity set for $T$. If not, then we can find $x,y\in X$ such that $x(1_G)\neq y(1_G)$, yet \[T(x)\upharpoonright M^2=T(y)\upharpoonright M^2.\] By strong irreducibility of $X$, we can assume that $x\sim y$. Indeed, let $U\subseteq G$ be a finite memory set for $T$ and $V\subseteq G$ a finite irreducibility set for $X$. Then we can find $x',y'\in G$ such that $x'\upharpoonright M^2\cdot U=x\upharpoonright M^2\cdot U$ and $y'\upharpoonright M^2\cdot U=y\upharpoonright M^2\cdot U$ and $x'\upharpoonright (G\smallsetminus M^2\cdot U\cdot V)=y'\upharpoonright (G\smallsetminus M^2\cdot U\cdot V)$. It follows that $x'\sim y'$ and $T(x')\upharpoonright M^2=T(x)\upharpoonright M^2$, $T(y')\upharpoonright M^2=T(y)\upharpoonright M^2$.

Since $\{g\in G\colon T(x)(g)\neq T(y)(g)\}\subseteq G\smallsetminus M^2$, by Observation~\ref{obs:postsurjectivity}, we get that $\{g\in G\colon x(g)\neq y(g)\}\subseteq G\smallsetminus M$, which contradicts that $x(1_G)\neq y(1_G)$.\\

The `in particular' part follows by applying also Lemma~\ref{lem:post-surjective-surjective}.
\end{proof}
\begin{corollary}
Let $T\colon A^G\rightarrow A^G$ be post-surjective and pre-injective. Then it is injective and also an isomorphism.
\end{corollary}
\begin{proof}
It follows immediately from Theorem~\ref{thm:isom} and Observation~\ref{obs:postsurjectivity}.
\end{proof}

\section{More observations on injectivity and post-surjectivity}
Let us have a closer look at the tight connection between injectivity and post-surjectivity. 

The following lemma establishes the connection between the injectivity and post-surjectivity sets.
\begin{lemma}\label{lem:injectivitypostsurjectivity}
Let $T\colon A^G\rightarrow A^G$ be an injective and post-surjective cellular automaton. 
\begin{itemize}
    \item Let $N$ be a symmetric injectivity set for $T$. Then it is also a post-surjectivity set.
    \item Let $M$ be a symmetric post-surjectivity set for $T$. Then it is also an injectivity set.
\end{itemize}
\end{lemma}
\begin{proof}
Let $N$ be a symmetric injectivity set (note that if $N$ is an arbitrary injectivity set, then $N\cup N^{-1}$ is symmetric and still an injectivity set). Suppose that it is not a post-surjectivity set. Then there exist $x,z\in A^G$ such that $\{g\in G\colon T(x)(g)\neq z(g)\}=\{1_G\}$, yet for every $y\sim x$ such that $T(y)=z$ we have $\{g\in G\colon x(g)\neq y(g)\}\not\subseteq N$. Choose such $y\sim x$ using post-surjectivity of $T$. There exists $g\notin N$ such that $x(g)\neq y(g)$. However, since $N$ is a symmetric injectivity set for $T$, we get that there is $h\in g\cdot N$ such that $T(x)(h)\neq T(y)(h)$. Since $h\neq 1_G$, this is a contradiction.

Conversely, suppose we are given a symmetric post-surjectivity set $M$. Suppose that it is not an injectivity set. Then there are $x,y\in A^G$ and $h\in G$ such that $x(h)\neq y(h)$, yet $T(x)\upharpoonright M=T(y)\upharpoonright M$. Without loss of generality, we may assume that $x\sim y$. Then $D=\{g\in G\colon T(x)(g)\neq T(y)(g)\}$ is also finite and $D\cap M=\emptyset$. However, by post-surjectivity, $\{g\in G\colon x(g)\neq y(g)\}\subseteq D\cdot M$. Since $h\notin D\cdot M$, we reach a contradiction.
\end{proof}

In the next lemma, we show the connection between the injectivity sets for $T$ and the memory sets for $T^{-1}$.
\begin{lemma}\label{lem:injectivity-memory}
Let $T\colon A^G\rightarrow A^G$ be an injective cellular automaton with a finite injectivity set $N$. Then $N$ is a memory set for $T^{-1}\colon T[A^G]\rightarrow A^G$.
\end{lemma}
\begin{proof}
Let $X=T[A^G]$ and for every subset $Z\subseteq G$ let $X_Z$ be the set of restrictions of elements from $X$ on the set $Z$, i.e. $X_Z=\{x\upharpoonright Z\colon x\in X\}$. For $a\in A$, let $O_a$ be the basic clopen set $\{x\in A^G\colon x(1_G)=a\}$ and set $X_a=X_N\cap T[O_a]$. We have $X_N=\coprod_{a\in A} X_a$, i.e. $X_N=\bigcup_{a\in A} X_a$, and for $a\neq b\in A$, $X_a\cap X_b=\emptyset$. Indeed, the former is clear; for the latter, if for some $a\neq b\in A$ we had $X_a\cap X_b\neq \emptyset$, then there would be $x\in O_a$ and $y\in O_b$ such that $T(x)\upharpoonright N=T(y)\upharpoonright N$. That would contradict that $N$ is an injectivity set for $T$ since $x(1_G)\neq y(1_G)$ implies that for some $n\in N$, $T(x)(n)\neq T(y)(n)$.

Now we define $\tau\colon X_N\rightarrow A$ as follows. For $a\in A$ and $p\in X_N$ we set \[\tau(p)=a\;\text{ if and only if }\; p\in X_a.\] Since $X_N$ is the disjoint union of $(X_a)_{a\in A}$, this is well-defined. Let $V\colon X\rightarrow A^G$ be a cellular automaton defined by $\tau$. Let us check that $V=T^{-1}$. Since $V$ is $G$-equivariant, it suffices to check that for any $x\in A^G$, $x(1_G)=V\circ T(x)(1_G)$. Set $a=x(1_G)$. Then $x\in O_a$ and $T(x)\upharpoonright N\in X_a$, so $V\circ T(x)(1_G)=\tau(T(x)\upharpoonright N)=a$, and we are done.
\end{proof}

The finer analysis of injectivity and post-surjectivity will now have the following application. Given an injective CA $T\colon A^G\rightarrow A^G$, provided that we know some injectivity sets $N$, resp. $M$, for $T$, resp. for $T^{-1}$, in order to show that $T$ is surjective it suffices to verify that every pattern from $A^{MN}$ is in the image of $T$. Second, we show that the image of every injective CA is a subshift of finite type, and additionally we have some quantitative information about the size of forbidden patterns.

Let us just recall here that a \emph{GOE pattern} (Garden of Eden pattern) for $T\colon A^G\rightarrow A^G$ is a pattern $p\colon D\rightarrow A$ such that for no $f\in T[A^G]$, $f\upharpoonright D=p$. It is a basic application of compactness that if $T$ is not surjective, then there is a non-trivial GOE pattern for $T$.

\begin{theorem}\label{thm:injectivity-finitesubshift}
Let $T\colon A^G\rightarrow A^G$ be an injective CA. Let $N\subseteq G$ be a finite injectivity set for $T$ containing $1_G$ and $M\subseteq G$ be a finite injectivity set for $T^{-1}\colon T[A^G]\rightarrow A^G$ containing $1_G$ (which exists by Proposition~\ref{prop:unifinject}).
\begin{enumerate}
    \item \label{thm:inject-item1} Then $T[A^G]\subseteq A^G$ is a subshift of finite type, and moreover, the forbidden patterns are defined on $MN$.
    \item \label{thm:inject-item2} More generally, denote by $X^n$ the closed invariant subshift $T^n[A^G]$, where $n\in\Nat$. Then $X^n$ is a subshift of finite type whose forbidden patterns are defined on $MN^n$.
    \item \label{thm:inject-item3} If $T$ is not surjective, then there exists a GOE pattern for $T$ supported on $MN$.
\end{enumerate}

\end{theorem}
\begin{proof}
We start with \eqref{thm:inject-item1}. Let $N$ be a finite injectivity set for $T$ containing $1_G$. If $T$ is surjective, there is nothing to prove. So suppose that it is not and set $X=T[A^G]$, and more generally, set $X^n=T^n[A^G]$, for $n\in\Nat$. So $X^1=X$ and we may also use the notation $X^0=A^G$. By Lemma~\ref{lem:injectivity-memory}, $N$ is a memory set for $T^{-1}\colon X\rightarrow A^G$, thus $T^{-1}$ on $X$ is defined by some $\tau\colon X_N\rightarrow A$, where $X_N=\{x\upharpoonright N\colon x\in X\}$. Notice that $T^{-1}$, defined by the same $\tau\colon X_N\rightarrow A$, when restricted to $X^n$, for $n\geq 1$, is a continuous bijection from $X^n$ onto $X^{n-1}$ having the same injectivity set $M$. Set $D:=MN$ and set $F=A^D\smallsetminus X_D$. 

We claim that $F$ is a finite set of forbidden patterns defining $X$. Let us denote by $Y$ the subshift of $A^G$ defined by forbidden patterns from $F$. Let $x\in X$. It is clear that for every $g\in G$, $g^{-1}\cdot x\upharpoonright D\notin F$, so $X\subseteq Y$.

Conversely, let $y\in Y$ and let us show that $y\in X$. As $N\subseteq D$, we have that for every $g\in G$, $g^{-1}\cdot y\upharpoonright N\in X_N$. Since $T^{-1}$ is defined by $\tau\colon X_N\rightarrow A$, we get that $T^{-1}(y)$ is defined. Set $z:=T\circ T^{-1}(y)$. If we prove that $z=y$ then we are done since $z=T(T^{-1}(y))\in X$. We need to check that for every $g\in G$, $y(g)=z(g)$. We do it for $g=1_G$, the same argument then works for any $g$. Suppose that $y(1_G)\neq z(1_G)$. Since $y\upharpoonright D\notin F$, there is $y'\in X$ such that $y\upharpoonright D=y'\upharpoonright D$. Since $M$ is an injectivity set for $T^{-1}$ and $y'(1_G)\neq z(1_G)$ there exists $h\in M$ so that $T^{-1}(z)(h)\neq T^{-1}(y')(h)$. Since clearly $T^{-1}(y)(h)=T^{-1}(y')(h)$, as $N$ is a memory set for $T^{-1}$, we have $T^{-1}(z)(h)\neq T^{-1}(y)(h)$. Then since $N$ is an injectivity set for $T$ we get that there is $g\in MN=D$ such that $T\circ T^{-1}(z)(g)\neq T\circ T^{-1}(y)(g)$. However, we claim that $T\circ T^{-1}(z)=T\circ T^{-1}(y)$, and this contradiction will finish the proof. Indeed, we have \[T\circ T^{-1}(z)=T\circ T^{-1}\circ T\circ T^{-1}(y)=T\circ T^{-1}(y),\] where we used that $T^{-1}\circ T$ is the identity on $A^G$.
\bigskip

We continue with \eqref{thm:inject-item2}. We shall prove the statement by induction. For $n=1$ this has been proved in \eqref{thm:inject-item1}. Suppose that $n>1$ and the statement has been proved for $n-1$. Set $D_n=MN^n$ and $F_n=A^{D_n}\smallsetminus X^n_{D_n}$, and let $Y_n$ be the subshift of $A^G$ defined by forbidden patterns from $F_n$ (so $D=D_1$, $F=F_1$, and $Y=Y_1$). We claim that $X^n=Y_n$. Again, it is clear that $X^n\subseteq Y_n$, so we show the other inclusion. Pick $y\in Y_n$. We claim that
\begin{itemize}
    \item $T^{-1}(y)\in X^{n-1}$ and
    \item $T\circ T^{-1}(y)=y$.
\end{itemize}
This will together imply that $y\in X^n$. For the former, first notice that $T^{-1}(y)$ is well-defined. By the induction hypotheses, it suffices to show that for every $g\in G$, $T^{-1}(g^{-1}y)\upharpoonright MN^{n-1}\notin F_{n-1}$. We do it for $g=1_G$, for other elements it is analogous. Since $y\upharpoonright MN^n\notin F$ there exists $y'\in X^n$ such that $y\upharpoonright MN^n=y'\upharpoonright MN^n$. Then $T^{-1}(y')\in X^{n-1}$, therefore $T^{-1}(y')\upharpoonright MN^{n-1}\notin F_{n-1}$. Since $N$ is a memory set for $T^{-1}$, $T^{-1}(y)\upharpoonright MN^{n-1}=T^{-1}(y')\upharpoonright MN^{n-1}$, so indeed $T^{-1}(y)\upharpoonright MN^{n-1}\notin F_{n-1}$.

To show the latter, set $z:= T\circ T^{-1}(y)$. As in $\eqref{thm:inject-item1}$ we need to show that $z=y$, and by $G$-equivariance it suffices if we show that $z(1_G)=y(1_G)$. The argument is as in \eqref{thm:inject-item1}: If not, then since $M$ is an injectivity set for $T^{-1}$ there is $h\in M$ so that $T^{-1}(y)(h)\neq T^{-1}(z)(h)$, and so since $N$ is an injectivity set for $T$ there is $g\in MN$ so that \[T\circ T^{-1}(y)(g)\neq T\circ T^{-1}(z)(g)=T\circ T^{-1}\circ T\circ T^{-1}(y)(g)=T\circ T^{-1}(y)(g),\] where we used that $T^{-1}\circ T$ is the identity. This contradiction finishes the proof of \eqref{thm:inject-item2}.
\bigskip

We finish with \eqref{thm:inject-item3}. Again denote $MN$ by $D$, $T[A^G]$ by $X$, and by $Y_F$ the subshift of $A^G$ whose forbidden patterns are from $A^D\smallsetminus X_D$. It follows from the proof of \eqref{thm:inject-item1} that $T^{-1}\colon  Y_F\rightarrow A^G$ is injective and inverse to $T$. If there is no GOE pattern supported on $D$, then $Y_F=A^G$ and so $T^{-1}$ is defined on the whole $A^G$, and moreover it is inverse to $T$. That contradicts that $T$ is not surjective.
\end{proof}
\section{Dual surjunctive groups and ultraproducts}

Let us recall the following conjecture due to Gottschalk \cite{gott}.

\begin{conjecture} \label{con:sur}
Suppose $T\colon A^G \to A^G$ is $G$-equivariant continuous injective map, that is, an injective cellular automaton. Then $T$ is surjective and hence an isomorphism.
\end{conjecture}

A group $G$ is called \emph{surjunctive} if Conjecture \ref{con:sur} is true for $G$ and any finite $A$. The class of surjunctive groups is closed under subgroups \cite[Lemma 1.1]{weiss} and ultraproducts \cite[Theorem 3]{gleb3}. All sofic groups are surjunctive \cite{grom, weiss}.

A group $G$ is \emph{dual surjunctive} if every post-surjective cellular automaton $T\colon A^G\to A^G$ is pre-injective and hence is an isomorphism by Theorem \ref{thm:isom}. All sofic groups are dual surjunctive \cite[Theorem 2]{cap}.

One may introduce another classes of groups.
\begin{definition}
We call a group $G$ \emph{s-surjunctive} if for any finite set $A$ and any strongly irreducible subshift of finite type $X\subseteq A^G$, every injective CA $T\colon X\rightarrow X$ is surjective.

Analogously, we call a group $G$ \emph{dual s-surjunctive} if for any finite set $A$ and any strongly irreducible subshift of finite type $X\subseteq A^G$, every strongly post-surjective CA $T\colon X\rightarrow X$ is pre-injective (and hence an isomorphism by Theorem~\ref{thm:isom}).\\

Moreover, we introduce the notions of \emph{ss-surjunctivity} and \emph{ss-dual surjunctivity}, which are defined as s-surjunctivity, resp. s-dual surjunctivity, just the subshift $X\subseteq A^G$ in the definition is required to be only strongly irreducible, not necessarily of finite type. 
\end{definition}
We recall here that the \emph{Myhill property} of a subshift or a more general dynamical system is the property that pre-injectivity of a continuous $G$-equivariant map implies its surjectivity, and the \emph{Moore property} is the converse, i.e. surjectivity implies pre-injectivity.

\noindent{\bf Examples.}
\begin{enumerate}
    \item Every amenable group is ss-surjunctive (and therefore also s-surjunctive). This follows from \cite{csc12}, where the authors prove the Myhill property for amenable groups and strongly irreducible subshifts.
    \item Every amenable group is both s-surjunctive and s-dual surjunctive. This follows from \cite{Fio03}, where the author shows the Garden of Eden theorem for amenable groups and strongly irreducible subshifts of finite type.
\end{enumerate}

 On the other hand, the existence of ss-dual surjunctive groups is more delicate. Clearly, every finite group is ss-dual surjunctive and so also every locally finite group is ss-dual surjunctive. We conjecture that every $G$ that contains $\Int$ as a subgroup is not ss-dual surjunctive. Fiorenzi in \cite[Section 3]{Fio00} (see also \cite[Exercise 5.49]{csc_book}) shows that the Moore property does not hold for $\Int$ and strongly irreducible subshifts. Her example might be also a counterexample disproving ss-dual surjunctivity for $\Int$ and groups containing $\Int$.

\begin{problem}
Are sofic groups s-surjunctive and s-dual surjunctive?
\end{problem}

Now we prove that the class of dual surjunctive groups is closed under subgroups and ultraproducts. The techniques used in the proof can be also applied to get a shorter proof of the result of Glebsky and Gordon from \cite{gleb3} that surjunctive groups are closed under ultraproducts.

\begin{lemma}
Let $H$ be a subgroup of a dual surjunctive group $G$. Then $H$ is dual surjunctive.
\end{lemma}
\begin{proof}
Let $T\colon A^H\rightarrow A^H$ be a post-surjective CA. The map on a memory set for $T$ also defines a CA $T'\colon A^G\rightarrow A^G$ which is also post-surjective by \cite[Proposition 4]{cap}. Therefore, since $G$ is dual surjunctive, $T'$ is pre-injective, and it easily follows that $T$ is pre-injective as well.
\end{proof}
\begin{theorem}\label{thm:dualsurjultraprod}
Let $(G_n)_{n\in\N}$ be a sequence of dual surjunctive groups. Let $\U$ be an ultrafilter on $\Nat$. Then the ultraproduct $\prod_\U G_n$ is dual surjunctive as well.
\end{theorem}
\begin{proof} 
Fix an ultrafilter $\U$ and assume without loss of generality that it is non-principal. Denote the ultraproduct $\prod_\U G_n$ by $\G$. Let $T\colon A^\G\rightarrow A^\G$ be a post-surjective continuous $\G$-equivariant map. $T$ is given by a map $\tau\colon A^F\rightarrow A$, where $F\subseteq \G$ is, without loss of generality, a finite symmetric memory set, which is also a post-surjectivity set for $T$. By \L o\'{s}'s theorem, we can find finite symmetric sets $F_n\subseteq G_n$ (with $|F_n|=|F|$) and maps $\tau_n\colon A^{F_n}\rightarrow A$ such that $F=\prod_\U F_n$ and $\tau=\prod_\U \tau_n$. For each $n$ we can then define a continuous $G_n$-equivariant map $T_n\colon A^{G_n}\rightarrow A^{G_n}$ using $\tau_n$.

For a sequence $(c_n)_{n\in\N}$, where $c_n\in A^{G_n}$, we denote by $(c_n)_\U$ the element $c=\prod_\U c_n$. That is, the element $c\in A^\G$ such that for each $g\in \G$, represented by a sequence $(g_n)_{n\in\N}$, where $g_n\in G_n$, \[c(g)=a\text{ if and only if }\forall^\U n\; (c_n(g_n)=a).\] Denote by $\mathcal{I}$ the subset of $A^\G$ consisting of elements of the form $(c_n)_\U$. It is straightforward to verify that $\mathcal{I}\subseteq A^\G$ is a dense subset which is invariant under the action of $\G$ and under the relation $\sim$.
\medskip

\begin{claim}
For $\U$-many $n$, $F_n$ is a post-surjectivity set for $T_n$. In particular, for $\U$-many $n$, $T_n$ is post-surjective.
\end{claim}

Indeed, otherwise we get for $\U$-many $n$, elements $c_n$ and $e_n\sim d_n:=T_n(c_n)$ such that there is no $c'_n\sim c_n$ satisfying $\Delta(c_n,c'_n)\subseteq \Delta(d_n,e_n)\cdot F_n$. Clearly, without loss of generality, we can assume that for $\U$-many $n$ we have $\Delta(d_n,e_n)=\{1_{G_n}\}$. 

Then we have
\begin{itemize}
\item $T((c_n)_\U)=(d_n)_\U$,
\item $\Delta((d_n)_\U,(e_n)_\U)=\{1_\G\}$.
\end{itemize}

Since $T$ has $F$ as a post-surjectivity set, we can find $c'\in A^\G$ such that $T(c')=(e_n)_\U$ and $\Delta(c',(e_n)_\U)\subseteq F$. Since the set $\mathcal{I}$ is invariant under the relation $\sim$, we have $c'\in\mathcal{I}$ and we can find elements $c'_n\in A^{G_n}$, for each $n$, so that $c'=(c'_n)_\U$. It easily follows that for $\U$-many $n$, $\Delta(c'_n,c_n)\subseteq F_n$ and $T_n(c'_n)=e_n$, a contradiction. This finishes the proof of the claim.\\

By our assumption, that the groups $G_n$ are dual surjunctive, it follows that for $\U$-many $n$, $T_n\colon  A^{G_n}\rightarrow A^{G_n}$ is pre-injective.
\bigskip

Now suppose that $T$ is not pre-injective. This means that there are elements $c\sim d\in A^\G$ such that $c\neq d$ and $T(c)=T(d)$. Denote by $D$ the non-empty finite set $\Delta(c(g)\neq d(g))$. Again by \L o\'{s}'s theorem, we can find non-empty finite sets $D_n\subseteq G_n$ so that $D=\prod_\U D_n$.

Now since $\mathcal{I}$ is dense in $A^\G$ we can find nets of sequences $\{(c_n^\alpha)_n\colon \alpha\in S\}$ and $\{(d_n^\alpha)_n\colon \alpha\in S\}$, where $S$ is some index set and we have
\begin{itemize}
\item $(c_n^\alpha)_\U\to c$ and $(d_n^\alpha)_\U\to d$,
\item for every $\alpha\in S$ and for $\U$-many $n$, $\Delta(c_n^\alpha, d_n^\alpha)=D_n$.
\end{itemize}

It follows that for every $\alpha\in S$ and $\U$-many $n$, since $T_n$ is pre-injective with memory set $F_n$, that we have \[\emptyset\neq \Delta(T_n(c_n^\alpha), T_n(d_n^\alpha))\subseteq D_n\cdot F_n.\]
Consequently, we get that for every $\alpha\in S$ \[\emptyset\neq \Delta(T((c_n^\alpha)_\U), T((d_n^\alpha)_\U))\subseteq D\cdot F.\]

By compactness, by passing to a subnet if necessary, we can without loss of generality assume that there exists a non-empty finite set $E\subseteq D\cdot F$ such that for every $\alpha\in S$, \[\emptyset\neq \Delta(T((c_n^\alpha)_\U),T((d_n^\alpha)_\U))=E.\]
Since $(c_n^\alpha)_\U\to c$ and $(d_n^\alpha)_\U\to d$, and also $T((c_n^\alpha)_\U)\to T(c)$ and $T((d_n^\alpha)_\U)\to T(d)$, we obtain that  \[\emptyset\neq \Delta(T(c),T(d))=E,\] in particular $T(c)\neq T(d)$. This contradiction finishes the proof of the theorem.
\end{proof}

The previous result has as a corollary a topological description of the dual surjunctive groups in the space of marked groups. Let us define the background.

Let $S$ be a fixed finite set. One can topologize the set of (isomorphism classes of) groups with $S$ as a generating set as follows. First we identify each such a group $G$ with $F_S/N$, where $F_S$ is a free group on generators from $S$ and $N\triangleleft F_S$ is a normal subgroup. Then it suffices to notice that the set of normal subgroups of $F_S$ is a closed subset of $2^{F_S}$, therefore it is a compact metrizable space (homeomorphic to the Cantor space). Let us denote this space by $X_S$ (see \cite{mar, isomar} and the references therein).

It is known that for a fixed finite set $S$, the set of surjunctive groups is closed in $X_S$ (see \cite[Section 3.7]{csc_book}, \cite[Corollary 1.3]{csc11} or \cite{grom,gleb3}). We prove an analogous result for the set of dual surjunctive groups.
\begin{theorem}\label{thm:closedinmarkedgrps}
For a fixed finite set $S$, the set of dual surjunctive groups is closed in the space of $S$-marked groups $X_S$.
\end{theorem}
\begin{proof}
Let $(N_i)_{i\in\Nat}$ be a sequence of normal subgroups of $F_S$ converging to a normal subgroup $N\triangleleft F_S$ such that for each $i\in\Nat$, $F_S/N_i$ is dual surjunctive. We prove that $G:=F_S/N$ is dual surjunctive. Pick a non-principal ultrafilter $\U$ and let $\G$ be the corresponding ultraproduct of $(F_S/N_i)_{i\in\Nat}$. By Theorem~\ref{thm:dualsurjultraprod}, $\G$ is dual surjunctive. The map from $G$ to $\G$ defined on the generating set $S$ by the diagonal map $s\to (s)_\U$ is clearly a monomorphism. So $G$ embeds as a subgroup of $\G$ and therefore it is dual surjunctive itself.
\end{proof}
\begin{corollary}\label{cor:fullyresiduallyDS}
Fully residually dual surjunctive groups are dual surjunctive.
\end{corollary}
\begin{proof}
Let $G$ be fully residually dual surjunctive. We may suppose it is finitely generated - by a finite generating set $S$. It is then easy to see that $G$ is a limit, in the space of marked groups $X_S$, of dual surjunctive groups. So it is dual surjunctive itself by Theorem~\ref{thm:closedinmarkedgrps}.
\end{proof}
Using a result of Arzhantseva and Gal \cite{ArG}, we can now obtain the same closure property that they have originally obtained for the class of surjunctive groups.
\begin{corollary}
Let $G$ be a semidirect product $H\ltimes F$, where $H$ is dual surjunctive and $F$ is a finitely generated residually finite group. Then $G$ is dual surjunctive.
\end{corollary}
\begin{proof}
By \cite[Theorem 1]{ArG}, it suffices to check that
\begin{enumerate}
    \item \label{it1} fully residually dual surjunctive groups are dual surjunctive,
    \item \label{it2} semidirect extensions of dual surjunctive groups by finite groups are dual surjunctive.
\end{enumerate}
\eqref{it1} has been proved in Corollary~\ref{cor:fullyresiduallyDS} and in order to prove \eqref{it2} it suffices to show that virtually dual surjunctive groups are dual surjunctive. We proceed as in \cite[Lemma 6]{ArG}. Let $H\leq G$ be such that $|G\colon H|<\infty$ and $H$ is dual surjunctive. Let $T\colon A^G\rightarrow A^G$ be post-surjective. $A^G$ as an $H$-shift is isomorphic to the shift $(A^{H\backslash G})^H$, and it is easy to check that the induced $H$-equivariant map $T'\colon (A^{H\backslash G})^H\rightarrow (A^{H\backslash G})^H$ is post-surjective. Therefore it is injective and so is $T$.
\end{proof}

\begin{lemma}
Let $G$ be a group and $N\lhd G$ be a normal subgroup which is finitely generated and residually finite. Assume that $G/N$ is (dual) surjunctive. Then $G/Z(N)$ is (dual) surjunctive as well.
\end{lemma}
\begin{proof}
The group $G$ acts by conjugation  $G\ni g\mapsto i_g\in\aut(N)$ on $N$. Consider \[F\colon G\to G/N\times\aut(N),\ \  F(g)=\left(g/N,i_g\right).\] 
That gives an embedding of $G/Z(N)$ into $G/N\times\aut(N)$ wich is (dual) surjunctive, as if $N$ is a finitely generated residually finite group, then $\aut(N)$ is residually finite.
\end{proof}

\begin{conjecture}
If $G$ is (dual) surjunctive then free product $\Int * G$ is (dual) surjuncive as well.
\end{conjecture}

One of the basic notion in model theory is that of an \emph{elementary equivalence}. Two structures $A$ and $B$ in a language $L$ are \emph{elementary equivalent} if they satisfy the same first order $L$-sentences, that is a theorem expressible in $L$ is true in $A$ if and only if it is true in $B$.

\begin{corollary}\label{cor:elementaryeqDS}
Suppose that $A$ and $B$ are elementary equivalent groups in a group theory language $\{\cdot\}$ and $A$ is dual surjunctive. Then $B$ is also dual surjunctive.
\end{corollary}
\begin{proof}
If $A$ and $B$ have same universal theory, then $B$ embeds into some ultrapower of $A$ and the conclusion follows from Theorem \ref{thm:dualsurjultraprod}.
\end{proof}

\section{Direct finiteness conjecture}

In late 1960s, I. Kaplansky proved \cite{kap} that if $K$ is a field of characteristic 0 and $G$ is any group, then the group ring $K[G]$ is \emph{directly finite}. Let us recall that a ring $R$ with $1$ is called \emph{directly finite} if for any $x,y\in R$, the condition $xy=1$ implies $yx=1$. \emph{Kaplansky's Direct Finiteness Conjecture} says that $K[G]$ is directly finite for any field $K$ and any group $G$.

This conjecture attracted a lot of attention recently in a more general case when $K$ is a division ring. See \cite{ara} for the proof when $G$ is a residually amenable group and also \cite{jus} for the computational approach in characteristic 0. The most general result on Kaplansky conjecture was established in \cite{sofic} (see also \cite[Corollary 1.4]{csc07}), when $G$ is a sofic group. All sofic groups are surjunctive and dual surjunctive. The main idea of the proof in \cite{sofic} is to construct an embedding of $K[G]$ into simple continuous von Neumann regular ring. 

We give below an elementary proof of this conjecture when $G$ is surjunctive or dual surjunctive group and $K$ is an arbitrary field of positive characteristic. Our proof covers the case of sofic groups.

\begin{theorem} \label{thm:kap}
Surjunctive and dual surjunctive groups satisfy Kaplansky's Direct Finiteness conjecture for fields of positive characteristic.
\end{theorem}
\begin{proof}
Let us first assume that $K$ is a finite field. Observe that the group ring $K[G]$ is dense in $K^G$. Every element $a\in K[G]$ induces a continuous linear map $T_a$
\begin{equation}
T_a\colon K^G \to K^G,\ \  T_a(f) = f\ast a, \text{ where }(f\ast a)(x)=\sum_{y\in G}f(y)a\left(y^{-1}x\right)
\end{equation} 
which is $G$-equivariant, that is $T_a(g\cdot f)(x) = g\cdot T_a(f)(x)$. Moreover $T_b\circ T_a = T_{a b}$, for $a,b\in K[G]$. 

Suppose now $a b=1$ for some $a,b\in K[G]$. Then $T_b\circ T_a=T_{a b} = \id$ is the identity on $K[G]$ which is dense in $K^G$. Therefore it is the identity on $K^G$.

\begin{claim}
The map $T_b$ is post-surjective and $T_a$ is injective.
\end{claim}
\begin{proof}[Proof of Claim]
The injectivity of $T_a$ is clear, since $T_b\circ T_a=T_{a b}=\id$. We prove that $T_b$ is post-surjective. Suppose $c=T_b(e)$ and $c\sim d$. Then $T_a(c)\sim T_a(d)$, so $e\sim  e+T_a(d)-T_a(c)=:e'$ and
\begin{align*}
T_b(e') &= T_b\circ T_a(d)+T_b(e-T_a(c))=  T_{ab}(d) + T_b(e)-T_{ab}(c)\\ &=T_{ab}(d)+T_b(e)-T_{bab}(e)=d+T_b(e)-T_b(e)=d.
\end{align*}
Hence $T_b$ is post-surjective.
\end{proof}
Since $G$ is dual surjunctive (surjunctive respectively), $T_b$ is a bijective ($T_a$ is a bijective respectively) continuous map. Thus $T_a$ is the inverse of $T_b$, so $T_a\circ T_b = T_{ba}$ is the identity on $K^G$ as well. Hence $ba=1$.

Suppose $K$ is an arbitrary field of positive characteristic and $ab = 1$ but $ba\neq 1$ for some $a,b\in K[G]$. 
Let $R=\langle k_i, l_i : 1\leq i\leq n \rangle$ be a subring of $K$ generated as a subring by the coefficients of $a$ and $b$. Hence $R$ is a finitely generated domain. 

Since $ba\neq 1$, we may write \[ba = \sum_{j=1}^m p_j g_j\] where $0\neq p_j\in K$ and elements $g_j\in G$ are pairwise distinct. We may assume that $p_1\neq 0$ and $g_1\neq 1_G$, or $p_1\neq 1$ and $g_1=1_G$. There exists a maximal ideal $I\lhd R$ (see e.g. \cite[Lemma 3.2 (iv)]{nica}) such that $p_1\not\in I$. Then $F = R/I$ is a field, which is a finitely generated as a ring. Therefore $F$ is a finite field, as we are in positive characteristic. Let $f\colon R[G]\to F[G]$ be a quotient homomorphism. Then $1=f(ab)=f(a)f(b)$, but $f(ba)=f(b)f(a)\neq 1$, as $f(p_1 g_1) = f(p_1) g_1 \neq 0$. This finishes the proof.
\end{proof}

We prove that the class of groups satisfying Kaplansky Conjecture is closed under taking ultraproducts.

\begin{proposition}
Suppose $R$ is a ring and $(G_i)_{i\in I}$ is a family of groups such that $R[G_i]$ is directly finite for $i\in I$. Then $R[G]$ is directly finite, where $G = \prod_{i\in I}G_i/\U$ is an ultraproduct of $(G_i)_{i\in I}$.

In particular the class of groups satisfying Kaplansky's Direct Finiteness Conjecture is closed under taking ultraproducts.
\end{proposition}
\begin{proof}
Suppose $x,y\in R[G]$ and $xy=1$. That is $x=\sum_{j=1}^n c_j g_j$, $y=\sum_{k=1}^m c'_k g'_k$, where $c_j,c'_k\in R$, $g_j=(g_{i,j})_{i\in I}/U\in G$, $g'_k=(g'_{i,k})_{i\in I}/U\in G$ for some $g_{i,j}, g'_{i,k}\in G_i$ and $g_s\neq g_t$, $g'_p\neq g'_q$ for $1\leq s\neq t\leq n$ and $1\leq p\neq q\leq m$.

Consider $x_i = \sum_{j=1}^n c_j g_{i,j}$ and $y_i=\sum_{k=1}^m c'_k g'_{i,k}$ ($x_i,y_i\in R[G_i]$). In order to finish the proof, it is enough to prove the following claim.

\begin{claim}
$xy=1$ if and only if $\{i\in I : x_iy_i=1\}\in\U$.
\end{claim}
\begin{proof} Let \[I_1 = \left\{i\in I : \forall_{1\leq j,j'\leq n}\ \forall_{1\leq k,k'\leq m}\ \  g_j\cdot g'_k = g_{j'}\cdot g'_{k'} \Leftrightarrow g_{i,j}\cdot g'_{i,k} = g_{i,j'}\cdot g'_{i,k'} \right\},\]
\[I_2 = \left\{i\in I : \forall_{1\leq j\leq n}\ \forall_{1\leq k\leq m}\ \  g_j\cdot g'_k = e_G \Leftrightarrow g_{i,j}\cdot g'_{i,k} = e_{G_i} \right\}.\]

Clearly $I_1,I_2\in\U$. Moreover for $i\in I_1\cap I_2$, the canonical from of $x_iy_i$ in $R[G_i]$ has same coefficients as the canonical form of $xy$ in $R[G]$ (because $x\cdot y = \sum\sum c_jc'_k g_j\cdot g_{k'}$). Therefore for $i\in I_1\cap I_2$, $xy=1$ if and only if $x_iy_i=1$.
\end{proof}
\end{proof}

\section{Metric direct finiteness conjecture}

This section is devoted to a \emph{metric version of the Kaplansky conjecture}, which gives the standard Kaplansky conjecture on the level of \emph{metric ultraproduct}.

Let us give a couple of definitions. A \emph{metric group} $(G,\|\cdot\|)$ is a group equipped with a \emph{bi-invariant} (also called \emph{conjugation-invariant}) norm $\|\cdot\|\colon G\to [0,\infty)$ satisfying $\|e\|=0$, $\|gh\|\leq \|g\| + \|h\|$, $\left\| g^{-1} \right\| = \|g\| = \|hgh^{-1}\|$, $\|g\|=0$ if and only if $g=e$. 

Throughout this section we fix a metric group $(G,\|\cdot\|)$, a commutative ring $R$ with 1 and a family $\G = (G_n,\|\cdot\|_n)_{n\in \N}$ of metric groups such that \[\sup\{\|\cdot\|_n : n\in\N\}<\infty.\] A \emph{metric ultraproduct} $\G^*_{\text{met}}$ of $\G$ with regard to $\U$ is defined as a quotient group: 
\[\G^*_{\text{met}} = \frac{\prod_{n\in \N}G_n}{\E}, \text{ where } \E = \left\{\left(g_n\right)_{n\in\N}\in \prod_{n\in \N}G_n : \lim_{n\to \U}\left\|g_n\right\|_n = 0 \right\}.\]
 $\G^*_{\text{met}}$ is also a metric group. It is  equipped with a bi-invariant norm: 
\begin{equation} \label{eq:norm}
\|\cdot\|\colon \G^*_{\text{met}}\to [0,\infty)\text{ defined by }\left\|(g_n)/\E\right\| = \lim_{n\to \U} \left\|g_n\right\|_n.
\end{equation}

Consider a group ring $R[G]$. There is a ring homomorphism $\varepsilon\colon R[G]\to R$ called the \emph{augmentation map}, defined as $\varepsilon(f) = \sum_{g\in G}f(g)$. The \emph{augmentation ideal} $\Delta^R(G)$ (or just $\Delta(G)$) is the kernel of $\varepsilon$. It is a free $R$ module with a basis $\{ g-1 : g\in G\setminus\{e\}\}$ \cite[Section 3.3]{gr}. 

Consider $f \in R[G]$. The \emph{length} $l(f)$ of $f$ is $n$ if $f = \sum_{i=1}^n \lambda_i g_i$, where $\lambda_i\neq 0$ and $g_i\neq g_j$, for $1\leq i\neq j\leq n$.

We need a pseudonorm $\|\cdot\|_S$ on $\Delta(G)$, defined below.

\begin{definition}\label{def:l1} Suppose $f=\sum_{i=1}^n \lambda_i g_i\in \Delta^R(G)$, where $\lambda_i\neq 0$ and $g_i\neq g_j$ for $i\neq j$. Define 
\begin{enumerate}
\item $\supp(f) =\{g_1,\ldots,g_n\}$.
\item $\|f\|_S = \min\left\{\sum_{i=1}^N \left\|u_i^{-1}w_i\right\| : f=\sum_{i=1}^N p_i(u_i-w_i),\; u_i, w_i\in\supp(f),\  p_i\in R\setminus\{0\}\right\}.$
Notice that, $\|\cdot\|_S$ is indeed defined with minimum instead of infimum, since there are only finitely many such decomposition's of $f$.
\end{enumerate}
\end{definition}
 
We have $\|f\|_S=0$ if and only if $f=0$ and $\|f_1+f_2\|_S\leq \|f_1\|_S+\|f_2\|_S$, but we do not need this inequality in arguments below.

We give a criterion for $K[\G^*_{\text{met}}]$ to satisfy Kaplansky Conjecture below in Theorem \ref{thm:kapp} below. Let us first give a couple of observations. There is a natural surjective homomorphism $\prod_{n\in\N}G_n \to \G^*_{\text{met}}$ of groups, which gives surjective homomorphism of rings and an isomorphism
\[\frac{K\left[\prod_{n\in\N}G_n \right]}{I} \overset{\cong}{\longrightarrow} K\left[\G^*_{\text{met}}\right],\] for a two-sided ideal $I\lhd K\left[\prod_{n\in\N}G_n \right]$ generated by $\left\{1-(g_n) : (g_n)\in \E\right\}$, by \cite[Corollary 3.2.8 p. 132, Proposition 3.3.4 p. 136]{gr}. Observe that $\|1-(g_n) \|_S=0$ in $K\left[\prod_{n\in\N}G_n \right]$, for every $(g_n)\in \E$. In fact $I=\left\{f\in \triangle\left(\prod_{n\in\N}G_n \right): \|f\|_S=0\right\}$.

\begin{theorem}\label{thm:kapp}
Fix a finite commutative unital ring $K$ and a family $\G$ of metric groups. The following facts are equivalent.
\begin{enumerate}
    \item $K\left[\G^*_{\text{met}}\right]$ satisfies Kaplansky conjecture.
    \item For $\U$-almost all $n\in\N$, for every $\varepsilon>0$ and every natural $N\in\N$ there is $\delta(\varepsilon,N)>0$ such that for every $a,b \in K[G_n]$, $l(a), l(b)\leq N$, 
    \[\text{if } ab-1\in \Delta^K(G_n) \text{ and } \|ab-1\|_S<\delta(\varepsilon,N),\text{ then }\|ba-1\|_S<\varepsilon,\]
    where $\|\cdot\|_S$ is defined on $K[G_n]$ as in \ref{def:l1}(2).
\end{enumerate}
\end{theorem}

We need a lemma below on the coherence of the norm $\|\cdot\|_S$ defined in \ref{def:l1} on $K[\G^*_{\text{met}}]$ and on each $K[G_n]$.

\begin{lemma}\label{lem:lim} Fix a finite ring $R$ and $a\in\Delta^R(\G^*_{\text{met}})$, $a =p_1\cdot \bar{g}_1 + \ldots + p_{N}\cdot \bar{g}_N$, for $p_i\in R\setminus\{0\}$, where $\bar{g}_i=\left(g_{i,n}\right)/\E\in \G^*_{\text{met}}$, for $g_{i,n}\in G_n$.
%Suppose $\bar{g}_i\neq\bar{g}_j$ and $g_{i,n}\neq g_{j,n}$ hold for $1\le i\neq j\le N$. 
Define \[a_n = p_1\cdot g_{1,n} + \ldots + p_{N}\cdot g_{N,n}\in R[G_n].\] 
Then \[\|a\|_S = \lim_{n\to\U}\|a_n\|_S.\]
%where $\|\cdot\|$ is on $K[\G^*_{\text{met}}]$ and $\|\cdot\|_n$ is on $K[G_n]$, defined as in \ref{def:l1}.
\end{lemma}
\begin{proof}
$\ge$ Suppose that $\|a\|_S = \sum_{i=1}^N \left\|u_i^{-1}w_i\right\|$, where $a=\sum_{i=1}^N p_i(u_i-w_i)$ is a minimising decomposition of $a$ and $u_i=(u_{i,n})/\E$, $w_i=(w_{i,n})/\E$. For every $n$, let $K_n:=\sum_{i=1}^N \left\|u_{i,n}^{-1}w_{i,n}\right\|$. We have clearly that
\begin{itemize}
    \item $\|a\|_S=\lim_{n\to\U} K_n$,
    \item $\|a_n\|_S\leq K_n$.
\end{itemize}
Consequently, $\|a\|_S\ge \lim_{n\to\U}\|a_n\|_S$.

$\le$ For $\U$-many $n\in\N$ (as $R$ is finite), we have that $\|a_n\|_S=\sum_{i=1}^N \left\|u_{i,n}^{-1}w_{i,n}\right\|$ and $a_n=\sum_{i\leq m} p_i (u_{i,n}-w_{i,n})$, for some non-zero $\{p_1,\ldots,p_N\}$ from $R$. Set $K:=\sum_{i=1}^N \left\|u_{i}^{-1}w_{i}\right\|$, where $u_i=(u_{i,n})/\E$ and $w_i=(w_{i,n})/\E$. Then we clearly have
\begin{itemize}
    \item $K=\lim_{n\to\U} \|a_n\|_S$,
    \item $\|a\|_S\leq K$.
\end{itemize}
It follows that $\|a\|_S\le \lim_{n\to\U}\|a_n\|_S$.
\end{proof}

\begin{proof}[Proof of Theorem \ref{thm:kapp}]
$(2)\Rightarrow (1)$ Take $a,b\in K[\G^*_{\text{met}}]$ with $ab=1$ and $l(a),l(b)\leq N_0$. Then $ba-1\in \Delta^K(\G^*_{\text{met}})$, as $K$ is commutative and write
\begin{align*}
a =p_1\cdot \left(g_{1,n}\right)/\E + \ldots + p_{N_0}\cdot \left(g_{N_0,n}\right)/\E, \\
b =q_1\cdot \left(h_{1,n}\right)/\E + \ldots + q_{N_0}\cdot \left(h_{N_0,n}\right)/\E,
\end{align*}
for some $g_{i,n},h_{i,n}\in G_n$, $p_i,q_i\in K$. It is enough to prove that $\|ba-1\|_S=0$. Define elements $a_n,b_n\in K[G_n]$ as
\begin{align*}
a_n =p_1\cdot g_{1,n} + \ldots + p_{N_0}\cdot g_{N_0,n}, \\
b_n =q_1\cdot h_{1,n} + \ldots + q_{N_0}\cdot h_{N_0,n}.
\end{align*} 
Lemma \ref{lem:lim} implies that
\[ 0=\|ab-1\|_S = \lim_{n\to\U} \|a_n b_n-1\|_S \text{ and }\|ba-1\|_S = \lim_{n\to\U} \|b_na_n-1\|_S.
\]
The assumption (2) implies that $\|ba-1\|_S = \lim_{n\to\U} \|b_na_n-1\|_S=0$, so $ba=1$.

$(1)\Rightarrow (2)$ Suppose (2) fails. Then there are $\varepsilon_0>0$, $N_0\in\N$ and $\U$-many $n\in\N$, such that for all $\delta>0$ there are $a_{n,\delta}, b_{n,\delta}\in K[G_n]$, $l(a_{n,\delta}), l(b_{n,\delta}) \leq N_0$ satisfying 
\begin{equation}\label{assump}
a_{n,\delta}b_{n,\delta}-1\in \Delta^K(G_n) \text{ and }\|a_{n,\delta}b_{n,\delta}-1\|_S<\delta\text{ and }\|b_{n,\delta}a_{n,\delta}-1\|_S\geq \varepsilon_0.
\end{equation}
Observe that then $b_{n,\delta}a_{n,\delta}-1\in \Delta^K(G)$ (since $K$ is commutative).
Since $K$ is a finite field and $\U$ is a nonprincipal ultrafilter, we may assume that there are finite sequences $(p_1,p_2,\ldots,p_{N_0}),\ (q_1,q_2,\ldots,q_{N_0})\in K^{N_0}$
%$(r_1,r_2,\ldots,r_{2N_0}),\ (s_1,s_2,\ldots,s_{2N_0})\in K^{2 N_0}$ 
such that for $\U$-many $n\in\N$ and all $\delta>0$
\begin{align*}
a_{n,\delta}&= p_1\cdot g_{1,n,\delta} + \ldots + p_{N_0}\cdot g_{N_0,n,\delta}, \\
b_{n,\delta}&= q_1\cdot h_{1,n,\delta} + \ldots + q_{N_0}\cdot h_{N_0,n,\delta},
%a_{n,\delta}b_{n,\delta}-1&= r_1\cdot u_{1,n,\delta} + \ldots + r_{2N_0}\cdot u_{2N_0,n,\delta},\\
%b_{n,\delta}a_{n,\delta}-1&= s_1\cdot v_{1,n,\delta} + \ldots + s_{2N_0}\cdot v_{2N_0,n,\delta}.
\end{align*}
for some $g_{i,n,\delta},h_{i,n,\delta}\in G_n$. Define $a_\infty, b_\infty\in K\left[\G^*_{\text{met}}\right]$ as
\begin{align*}
a_{\infty}=p_1\cdot \bar{g}_1 + \ldots + p_{N_0}\cdot \bar{g}_{N_0}, \\
b_{\infty}=q_1\cdot \bar{h}_1 + \ldots + q_{N_0}\cdot \bar{h}_{N_0},
\end{align*}
where
\begin{align*}
    \bar{g}_i = \left(g_{i,n,\frac{1}{n}}\right)/\E,\ \ \bar{h}_i = \left(h_{i,n,\frac{1}{n}}\right)/\E\text{ for every }1\le i \le N_0.
\end{align*}
%\begin{equation}\label{eq:new}
%a_\infty = \left(a_{n,\frac{1}{n}}\right)/\E,\ b_\infty = \left(b_{n,\frac{1}{n}}\right)/\E,\ b_\infty a_\infty -1 = \left( b_{n,\frac{1}{n}}a_{n,\frac{1}{n}} - 1\right)/\E.
%\end{equation}
Then $a_{\infty}b_{\infty}-1\in \Delta^K(\G^*_{\text{met}})$ and $\|a_{\infty}b_{\infty}-1\|_S = 0$, by (\ref{assump}) and Lemma \ref{lem:lim}. Hence $a_{\infty}b_{\infty}=1$. Therefore $b_{\infty}a_{\infty}-1=0$ by the assumption (1). 
However (\ref{assump}) and Lemma \ref{lem:lim} imply that $\|b_\infty a_\infty -1\|_S\geq \varepsilon_0$, contradiction.
\end{proof}

We conjecture that condition (2) from Theorem \ref{thm:kapp} is true for any bounded class of metric groups. We call this \emph{Metric Kaplansky Conjecture}, which is a quantitative version of classical Kaplansky Conjecture. 

\begin{conjecture}\label{con:kapmet}
For any finite field $K$ every $\varepsilon>0$ and every natural $N\in\N$ there is $\delta(\varepsilon,N)>0$ such that for an arbitrary metric group $(G,\|\cdot\|)$, where $\|\cdot\|\le 1$, the following holds: for every $a,b \in K[G]$, such that $l(a), l(b)\leq N$ and $ab-1\in \Delta(G)$ 
    \[\text{if } \|ab-1\|_S<\delta(\varepsilon,N),\text{ then }\|ba-1\|_S<\varepsilon.\]
\end{conjecture}

It is known that Conjecture \ref{con:kapmet} holds for $\G_{\text{perm}} = (S_n,\frac{1}{n}\|\cdot\|_{\text{Hamming}})$, where $\|\cdot\|_{\text{Hamming}}$ is the Hamming norm on $S_n$, defined for $\tau\in S_n$ as $\|\tau\|_\text{Hamming} = |\supp(\sigma)|$ ($\supp(\sigma)=\{i : \sigma(i)\neq i\}$), as $\G^*_{\text{perm, met}}$ is a universal sofic group. Unfortunately, we do not have any concrete estimation for $\delta(\varepsilon,N)$. We actually conjecture that $\delta(\varepsilon,N)\le 2\varepsilon$ for $\G_{\text{perm}}$. 

\begin{comment}
\textcolor{red}{The $\max$-norm: \[\|f\|_{\max} = \max\left\{ \|g_i\| : 1\leq i\leq n \right\}\text{ and }\|0\|_{\max} = 0\] is well behaved, as the following is true: $\|f_1+f_2\|_{\max}\leq \max\{\|f_1\|_{\max},\|f_2\|_{\max}\}$; $\|f_1\cdot f_2\|_{\max} \leq \|f_1\|_{\max} + \|f_2\|_{\max}$; if $\|f\|_{\max}=0$, then $f=\lambda\cdot e$, for some $\lambda\in R$; in particular if $f\in \Delta^R(G)$ and $\|f\|_{\max}=0$, then $f=0$.}
\end{comment}

\section{Expansive dynamical systems}
In the last section, we consider dual surjunctivity for more general dynamical systems than subshifts. Here we follow and apply mainly the seminal results of Chung and Li \cite{NPLi}, and Li \cite{Li2019} on expansive algebraic actions.

Let $X$ be a compact metrizable space with some compatible metric $d$ which we may assume, without loss of generality, to be bounded by $1$. An action $\alpha\colon G\curvearrowright X$ of a group $G$ on $X$ by homeomorphisms is called \emph{expansive} if there exists $\delta$ such that for every $x\neq y\in X$ there is $g\in G$ such that $d(g\cdot x, g\cdot y)>\delta$. The real $\delta$ is then called the \emph{expansiveness constant} of $\alpha$.

Two elements $x,y\in X$ are called \emph{homoclinic} if $\lim_{g\to\infty} d(g\cdot x, g\cdot y)\to 0$. Clearly, homoclinicity is an equivalent relation which we shall denote by $\sim$. It coincides with the relation $\sim$ for the Bernoulli topological shifts. Note that since all compatible metrics on $X$ are uniformly equivalent, being expansive and homoclinic does not depend on the choice of the metric.

\begin{lemma}
Let $\alpha\colon G\curvearrowright X$ have expansiveness constant $\delta>0$. We assume that $G$ is countable. For two sequences $(x_n)_n, (y_n)_n\subseteq X$ we have that they both converge to a point $z\in X$ if and only if there is an increasing sequence of finite subsets $(F_n)_n$ of $G$ such that $\bigcup_{n\in\Nat} F_n=G$ such that for every $n\in\Nat$ and $f\in F_n$ we have $d(f\cdot x_n, f\cdot y_n)\leq \delta$.
\end{lemma}
\begin{proof}
Suppose that $(x_n)_n$ and $(y_n)_n$ both converge to $z\in X$. For every $g\in G$, $\lim_{n\to\infty} d(g\cdot x_n, g\cdot y_n)=0$, so there is $n_g\in\Nat$ such that for all $n\geq n_g$, $d(g\cdot x_n, g\cdot y_n)\leq \delta$. For each $n\in \Nat$ we set $F_n:=\{g\in G\colon n\geq n_g\}$. This gives us the desired sequence of finite subsets of $G$.

Conversely, assume that we have such a sequence of finite sets $(F_n)_n$ for the two sequence $(x_n)_n$ and $(y_n)_n$. Assume that they do not converge to a common point. We may assume that $(x_n)_n$ converges to some $x\in X$, $(y_n)_n$ converges to some $y\in X$, and $x\neq y$. By expansiveness, there exists $g\in G$ such that $d(g\cdot x, g\cdot y)>\delta$. It follows that for all sufficiently large $n\in\Nat$ we have $d(g\cdot x_n,g\cdot y_n)>\delta$, which contradicts that there is $n\in\Nat$ such that $g\in F_n$.
\end{proof}
Having the relation `$\sim$' at our disposal, we can define the notion of strong post-surjectivity in the same way as for subshifts.
\begin{definition}
Let $\alpha\colon G\curvearrowright X$ be given. Let $T\colon X\rightarrow X$ is a continuous $G$-equivariant map. We say that $T$ is \emph{strongly post-surjective} if there is a finite subset $F\subseteq G$ such that for every $x,y\in X$ such that $T(x)\sim y$, i.e. in particular the set $D:=\{g\in G\colon d(g\cdot T(x), g\cdot y)>\delta\}$ is finite, there exists $z\sim x$ such that $T(z)=y$ and $\{g\in G\colon d(g\cdot x, g\cdot z)>\delta\}\subseteq FD$.

The finite set $F$ is called the \emph{post-surjectivity set} for $T$.
\end{definition}
Obviously, this definition can only be reasonable provided the reversible maps satisfy it. We show this is indeed the case.

First, we need an analogue of Proposition~\ref{prop:unifinject}.
\begin{lemma}\label{lem:unifinjectivityexpansive}
Let $\alpha\colon G\curvearrowright X$ have expansiveness constant $\delta>0$. A continuous $G$-equivariant map $T$ is injective if and only if there exists a finite set $F\subseteq G$ such that for every $x\neq y\in X$ with $d(x,y)>\delta$ there is $f\in F$ such that $d(f\cdot T(x), f\cdot T(y))>\delta$.
\end{lemma}
\begin{proof}
Suppose that a continuous $G$-equivariant map $T$ satisfies such a condition. We show it is injective. Choose $x\neq y\in X$. By expansiveness, there is $g\in G$ such that $d(g\cdot x,g\cdot y)>\delta$, so there must be $f\in F$ such that $d(fg\cdot T(x),fg\cdot T(y))>\delta$; in particular, $T(x)\neq T(y)$.

We now show the converse. Suppose it does not satisfy the condition. Then for every finite set $F\subseteq G$ there are $x_F,y_F\in X$ such that $d(x_F,y_F)>\delta$, yet $d(f\cdot T(x_F), f\cdot T(y_F))\leq \delta$ for all $f\in F$. Since $X$ is compact we may assume that the nets $(x_F)_F$ and $(y_F)_F$ converge to elements $x$ and $y$, respectively. We have $d(x,y)\geq \delta$, so $x\neq y$. If $T(x)\neq T(y)$, then by expansiveness there exists $f\in G$ such that $d(f\cdot T(x), f\cdot T(y))>\delta$. Then however $d(f\cdot T(x_F), f\cdot T(y_F))>\delta$ for all sufficiently large sets $F$ containing $f$. This contradictions shows that $T(x)=T(y)$, so $T$ is not injective.
\end{proof}
\begin{definition}
For an injective continuous $G$-equivariant map $T$, the finite set $F$ from Lemma~\ref{lem:unifinjectivityexpansive} is called the \emph{injectivity} set for $T$.
\end{definition}
\begin{proposition}
Let $\alpha\colon G\curvearrowright X$ be as above. Let $T\colon X\rightarrow X$ be a continuous $G$-equivariant map which is moreover reversible. Then $T$ is strongly post-surjective.
\end{proposition}
\begin{proof}
Since $T$ is an injective continuous $G$-equivariant map, let $F$ be its finite injectivity set for provided by Lemma~\ref{lem:unifinjectivityexpansive}. We may suppose that it is symmetric. We show it is a post-surjectivity set for $T$. Choose $x,z\in X$ such that $T(x)\sim z$. Set $y:=T^{-1}(z)$. Since $T^{-1}$ is, by assumption, continuous and $G$-equivariant, we have $x\sim y$, and obviously $T(y)=z$. Set $D:=\{g\in G\colon d(g\cdot T(x),g\cdot z)>\delta)\}$, which is finite. Suppose that there is $h\in G\smallsetminus FD$ such that $d(h\cdot x,h\cdot y)>\delta$. Then, since $F$ is an injectivity set for $T$, there must be $f\in F$ such that $d(fh\cdot T(x),fh\cdot z)>\delta$. Therefore $fh\in D$, so $h\in FD$, a contradiction.
\end{proof}
\begin{theorem}\label{thm:reversibilityexpansive}
Let $\alpha\colon G\curvearrowright X$ be as above. Suppose that there exists a dense class in the homoclinicity relation $\sim$. Let $T\colon X\rightarrow X$ be a pre-injective and strongly post-surjective continuous $G$-equivariant map. Then $T$ is injective.
\end{theorem}
\begin{proof}
Let $\delta>0$ be the expansiveness constant. Suppose that such $T$ is not injective, so there are $w_1\neq w_2\in X$ with $T(w_1)=T(w_2)$. By expansiveness, there is $g\in G$ such that $d(g\cdot w_1,g\cdot w_2)>\delta$, so without loss of generality, we can assume that $d(w_1,w_2)>\delta$. Let $F\subseteq G$ be a finite symmetric post-surjectivity set for $T$. By the assumption, there exists $v\in X$ whose equivalence class $\{v'\in X\colon v'\sim v\}$ is dense in $X$.  Therefore by the continuity of $T$ and of the group action, we can find $v_1\sim v$ and $v_2\sim v$, where $d(v_1,w_1)$ and $d(v_2,w_2)$ are small enough so that $d(v_1,v_2)>\delta$ and for every $f\in F$, $d(f\cdot T(v_1),f\cdot T(v_2))<\delta$. Since $T$ is pre-injective, $T(v_1)\neq T(v_2)$. So by expansiveness, there exists $g\in G$ so that $d(g\cdot T(v_1),g\cdot T(v_2))>\delta$. By assumption, $g\notin F$. Now apply the strong post-surjectivity to $x=v_1$ and $z=T(v_2)$. We have $g\in D:=\{h\in G\colon d(h\cdot T(x),h\cdot z)>\delta\}\cap F=\emptyset$. By the strong post-surjectivity, there exists $y\sim x=v_1$ with $T(y)=z=T(v_2)$ such that $D':=\{h\in G\colon d(h\cdot x, h\cdot y)>\delta\}\subseteq FD$. However, pre-injectivity of $T$ implies that $y=v_2$, so since $d(v_1,v_2)>\delta$, $1_G\in D'$. Since for every $f\in F$, $d(f\cdot T(v_1),f\cdot T(v_2))<\delta$, we have $F\cap D=\emptyset$. Since $F$ is symmetric it follows that $1_G\notin FD$. This contradiction finishes the proof.
\end{proof}
We shall need one more simple lemma, where we require that \emph{every} class in the homoclinicity relation is dense.
\begin{lemma}\label{lem:densehomoclinic}
Let $\alpha\colon G\curvearrowright X$ be as above and suppose that every class $[x]_\sim$ in the homoclinicity relation is dense. Then every strongly post-surjective continuous $G$-equivariant map $T\colon X\rightarrow X$ is surjective.
\end{lemma}
\begin{proof}
Pick an arbitrary $x\in X$ and we show that there is $y\in X$ such that $T(y)=x$. Let $d$ be a compatible metric on $X$. Since $\{z\in X\colon z\sim T(x)\}$ is dense in $X$, for every $n\in\Nat$ there is $z_n\sim T(x)$ with $d(z_n,x)<1/n$. By strong post-surjectivity, there is $y_n\in X$, for each $n\in\Nat$, such that $T(y_n)=z_n$. Their cluster point $y$ clearly satisfies $T(y)=x$.
\end{proof}
We now apply the previous results to expansive algebraic actions of amenable groups, thoroughly studied in the context of surjunctivity for example in \cite{NPLi} and \cite{Li2019}.

The following is the most important definition.
\begin{definition}
Let $\alpha\colon G\curvearrowright X$ be an expansive action of a group $G$ on a compact metrizable space $X$ by homeomorphisms. We say that $\alpha\colon G\curvearrowright X$ is \emph{dual surjunctive} if every continuous $G$-equivariant strongly post-surjective map $T\colon X\rightarrow X$ is reversible.
\end{definition}
In the sequel, we work with algebraic actions. That is, actions of countable groups on compact metrizable abelian groups by continuous automorphisms. By the Pontryagin duality, all such actions of a countable group $G$ are in one-to-one correspondence with countable modules over the group ring $\Int G$. We refer to \cite[Chapter 13]{KerrLi} for an introduction to expansive algebraic actions and the notions of entropy from the next result. We recall that a group is \emph{polycyclic-by finite} if it is obtained recursively in finitely many steps by the group extension operation, using at each step a finite or a cyclic group.
\begin{theorem}\label{thm:expansivealgactionDS}
Let $\alpha\colon G\curvearrowright X$ be an expansive algebraic action of a countable amenable group on a compact metrizable abelian group $X$ with completely positive entropy with respect to the normalized Haar measure on $X$. Suppose that at least one of the following conditions is satisfied:
\begin{enumerate}
    \item\label{it:exp1} either $G$ is polycyclic-by-finite,
    \item\label{it:exp2} or the set $\Delta(X)$ of elements of $X$ that are homoclinic to the identity element $e_X$ of $X$ is dense in $X$.
\end{enumerate}
 Then $\alpha$ is dual surjunctive.
\end{theorem}
\begin{proof}
We need \eqref{it:exp2}. If \eqref{it:exp1} is satisfied, i.e. $G$ is polycyclic-by-finite, then by \cite[Theorem 1.2]{NPLi}, the assumption on completely positive entropy implies that the set $\Delta(X)$ is dense in $X$; that is, \eqref{it:exp1} implies \eqref{it:exp2}. Notice also that each equivalence class in $\sim$ is a translate of $\Delta(X)$, so actually each equivalence class is dense. Thus by Lemma~\ref{lem:densehomoclinic}, $T$ is surjective. 

It follows that we can apply \cite[Theorem 1.2]{Li2019} to get that $T$ is pre-injective. Finally, we have all the ingredients to apply Theorem~\ref{thm:reversibilityexpansive} to obtain that $T$ is reversible.
\end{proof}

\bibliography{ca}

\bibliographystyle{alpha}

\end{document}